\documentclass{gtmon_a}
\pdfoutput=1

\proceedingstitle{The Zieschang Gedenkschrift}
\conferencestart{5 September 2007}
\conferenceend{8 September 2007}
\conferencename{Conference in honour of Heiner Zieschang}
\conferencelocation{Toulouse, France}

\editor{Michel Boileau}
\givenname{Michel}
\surname{Boileau}

\editor{Martin Scharlemann}
\givenname{Martin}
\surname{Scharlemann}

\editor{Richard Weidmann}
\givenname{Richard}
\surname{Weidmann}

\title{Free group automorphisms with many fixed points at infinity}

\author[Andr\'e J\"ager and Martin Lustig]{Andr\'e J\"ager\newline Martin Lustig}
\givenname{Andr\'e}
\surname{J\"ager}
\address{Louis Pasteur-Strasse 58\\60439 Frankfurt am
  Main\\Germany\vspace{3pt}\\\newline
Math\'ematiques (LATP)\\Universit\'e P C\'ezanne -- Aix
Marseille III\\\newline
Ave E Normandie-Niemen\\13397 Marseille 20\\France}
\email{Andre.Jaeger@ui-gmbh.de}
\urladdr{}

\givenname{Martin}
\surname{Lustig}
\email{Martin.Lustig@univ-cezanne.fr}
\urladdr{}

\volumenumber{14}
\issuenumber{}
\publicationyear{2008}
\papernumber{14}
\startpage{321}
\endpage{333}

\doi{}
\MR{}
\Zbl{}

\arxivreference{}

\keyword{free group}
\keyword{automorphism}
\keyword{fixed point at infinity}

\subject{primary}{msc2000}{20E36}
\subject{secondary}{msc2000}{57M05}

\received{27 June 2006}
\revised{5 July 2007}
\accepted{31 July 2007}
\published{29 April 2008}
\publishedonline{29 April 2008}
\proposed{}
\seconded{}
\corresponding{}
\version{}

\makeatletter
\def\cnewtheorem#1[#2]#3{\newtheorem{#1}{#3}[section]
\expandafter\let\csname c@#1\endcsname\c@thm}
\makeatother

\let\xysavmatrix\xymatrix
\def\xymatrix{\disablesubscriptcorrection\xysavmatrix}
\AtBeginDocument{\let\bar\wbar\let\tilde\wtilde\let\hat\what}
\makeop{ind}
\makeop{Aut}
\makeop{Out}
\makeop{Fix}
\makeop{Stab}
\def\rk{\mathrm{rk}}

\makeatletter
\def\cnewtheorem#1[#2]#3{\newtheorem{#1}{#3}[section]
\expandafter\let\csname c@#1\endcsname\c@theorem}
\makeatother

\newtheorem{theorem}{Theorem}[section]
\cnewtheorem{proposition}[theorem]{Proposition}
\cnewtheorem{lemma}[theorem]{Lemma}
\cnewtheorem{corollary}[theorem]{Corollary}
\cnewtheorem{claim}[theorem]{Claim}
\theoremstyle{remark}
\cnewtheorem{conjecture}[theorem]{Conjecture}
\cnewtheorem{question}[theorem]{Question}
\cnewtheorem{problem}[theorem]{Problem}
\cnewtheorem{definition}[theorem]{Definition}
\cnewtheorem{remark}[theorem]{Remark}
\cnewtheorem{example}[theorem]{Example}
\cnewtheorem{condition}[theorem]{Condition}
\cnewtheorem{Quote}[theorem]{Quote}
\cnewtheorem{aside}[theorem]{Aside}

\def\SS{{\mathcal S}}
\def\G{{\mathcal G}}

\def\bdry{\partial}

\begin{document}

\begin{htmlabstract}
A concrete family of automorphisms &alpha;<sub>n</sub> of the free group F<sub>n</sub>
is exhibited, for any n &ge; 3, and the following properties are
proved:  &alpha;<sub>n</sub> is irreducible with irreducible powers, has
trivial fixed subgroup, and has 2n-1 attractive as well as 2n
repelling fixed points at &part;F<sub>n</sub>.
As a consequence of a recent result of V Guirardel there can not be
more fixed points on &part;F<sub>n</sub>, so that this family provides the
answer to a question posed by  G Levitt.
\end{htmlabstract}

\begin{abstract}
A concrete family of automorphisms $\alpha_n$ of the free group $F_n$
is exhibited, for any $n \geq 3$, and the following properties are
proved:  $\alpha_n$ is irreducible with irreducible powers, has
trivial fixed subgroup, and has $2n-1$ attractive as well as $2n$
repelling fixed points at $\partial F_n$.
As a consequence of a recent result of V Guirardel there can not be
more fixed points on $\partial F_n$, so that this family provides the
answer to a question posed by  G Levitt.
\end{abstract}

\begin{asciiabstract}
A concrete family of automorphisms alpha_n of the free group F_n is
exhibited, for any n > 2, and the following properties are proved:
alpha_n is irreducible with irreducible powers, has trivial fixed
subgroup, and has 2n-1 attractive as well as 2n repelling fixed points
at bdry F_n.  As a consequence of a recent result of V Guirardel there
can not be more fixed points on bdry F_n, so that this family provides
the answer to a question posed by G Levitt.
\end{asciiabstract}

\maketitle

\section{Introduction}

Let $F_n$ be a free group of finite rank $ n \geq 2$. It is well known
that every automorphism $\alpha$ of $F_n$ induces a homeomorphism
$\bdry\alpha$ on the Gromov boundary $\bdry F_n\,$. Every fixed point
of $ \bdry \alpha $ is either attracting or repelling (= attracting
for $ \bdry \alpha^{-1}) $ or it belongs to $ \bdry \Fix(\alpha)$,
which embedds into $ \bdry F_n\,$, as the fixed subgroup $
\Fix(\alpha) = \{ w \in F_n \, |\, \alpha(w) = w \} $ is quasiconvex
in $F_n\,$.  Notice that $ \Fix(\alpha) $ acts on the set of
attracting fixed points $ \Fix^+(\bdry \alpha) $ of $ \bdry
\alpha$. After various proofs that $\Fix(\alpha)$ is finitely
generated and that $ \Fix^+(\bdry \alpha) / \Fix(\alpha) $ is finite
for all $\alpha \in \Aut(F_n)$ (see Gersten \cite{Ge}, Cooper
\cite{Co}, Goldstein and Turner \cite{GT}, Cohen and Lustig \cite{CL},
Paulin \cite{Pa}, Gaboriau, Levitt, and Lustig \cite{GLL} etc), the
following improvement of Bestvina and Handel's Theorem \cite{BH} (also
known as the the Scott Conjecture) has been given by Gaboriau, Jaeger,
Levitt and Lustig \cite{GJLL}:
$$ \rk(\Fix(\alpha))+{1\over 2}\# (\Fix^+(\bdry \alpha)/ \Fix(\alpha)) \
 \leq \ n $$
It follows in particular that, if $ \Fix(\alpha) $ is trivial, then the total
set of fixed points $ \Fix(\bdry \alpha) = \Fix^+(\bdry \alpha) \cup 
\Fix^+(\bdry \alpha^{-1}) $ at $\bdry F_n$ is finite and satisfies
$$\# \, \Fix(\bdry \alpha)\ \leq\ 4n.$$ It seems a natural question
(posed originally to us by G Levitt) to ask whether automorphisms
exist with trivial fixed subgroup which satisfy equality in this last
formula, and if not, what the best possible bound is. In particular,
one would like to know the answer to this question for the class of
{\it irreducible automorphisms $ \alpha $ with irreducible powers
(iwip)}, ie, $\alpha^t$ does not map any non-trivial proper free
factor of $ F_n\,$ to a conjugate of itself, for any $ t \geq 1$.
Since then, it has been shown by Guirardel \cite{Guir} (see also
Handel and Mosher \cite{HM}) that iwip automorphisms can never satisfy
equality, see \fullref{V.1}.

In view of these  results, this
paper gives an answer to Levitt's question. We consider the following 
family of automorphisms:
$$
\begin{array}{llll}
\alpha_n\co &F_n &\rightarrow & F_n \\  
&a_1&\mapsto & a_1 a_2 \ldots a_n \\
&a_2&\mapsto   &  a_2 a_1 a_2 \\
&a_3&\mapsto    & a_3 a_1 a_2 a_3 \\
& & \, \, \vdots \\
&a_n&\mapsto   &  a_n a_1 a_2 a_3 \ldots a_n . \\
\end{array}
$$

\begin{theorem}
\label{0.1} 
For any $n\geq 3 $ the automorphism $\alpha_n$ is
irreducible with irreducible powers, has trivial fixed subgroup, and has
precisely $4n-1$ distinct fixed points at $\bdry F_n\,$. Among these there are 
$2n-1$ attractive ones and $2n$ repelling ones.
The same is true for all positive powers of $\alpha_n\,$.
\end{theorem}

 The result and some related material will be
 discussed in the last section of this paper.
 Note also that an earlier version of this paper,
 containing already the main result,
 was ciculated as preprint in 1998.

\section{The attracting fixed points of $\bdry\alpha_n$}
\label{I.}

Consider the following set of $ 2n -1 $ infinite words, notice that
they are all positive or negative and hence reduced, and check that they
are fixed by $ \alpha_n$. Here a {\it positive} (or a {\it negative}) word
is a word in the given basis with only positive (or only negative)
exponents. Similarly, a {\it positive} automorphism of $ F_n $ is an
automorphism for which the image of a given basis consists entirely of
positive words in this basis. 
$$\begin{array}{lll}
X_1 &= &a_1 a_2 a_3 \ldots a_n \alpha_n(a_2 a_3 \ldots a_n) 
\alpha_n^2(a_2a_3 \ldots a_n) \alpha_n^3(a_2 a_3 \ldots a_n) \ldots \\
	 X_2 &= &a_2 a_1 a_2 \alpha_n(a_1 a_2)\alpha_n^2(a_1 a_2)\alpha_n^3(a_1a_2)
\ldots \\
	 X_3 &= &a_3 a_1 a_2 a_3 \alpha_n(a_1 a_2 a_3)  \alpha_n^2(a_1 a_2 a_3) 
\alpha_n^3(a_1 a_2 a_3) \ldots \\
	 X_4 &= & a_4 a_1 a_2 a_3 a_4 \alpha_n(a_1 a_2 a_3 a_4)  
\alpha_n^2(a_1 a_2a_3 a_4)  \alpha_n^3(a_1 a_2 a_3 a_4) \ldots \\
	     & \, \vdots \\
\end{array}
$$

$$\begin{array}{lll}
	 X_n &= & a_n a_1 a_2 a_3 \ldots a_n \alpha_n(a_1 a_2 a_3 \ldots a_n) 
\alpha_n^2(a_1 a_2 a_3 \ldots a_n)  \\ 
&&\alpha_n^3(a_1 a_2 a_3 \ldots a_n) \ldots \cr
	 Y_2 &= & a_2^{-1} a_1^{-1} a_2 ^{-1} \alpha_n(a_1^{-1} a_2 ^{-1}) 
\alpha_n^2(a_1^{-1} a_2 ^{-1})  \alpha_n^3(a_1^{-1} a_2 ^{-1}) \ldots       
\\
	 Y_3 &= & a_3^{-1} a_2^{-1} a_1^{-1} a_3 ^{-1} \alpha_n(a_2^{-1}
a_1^{-1} a_3 ^{-1})  \alpha_n^2(a_2^{-1} a_1^{-1} a_3 ^{-1}) \ldots \\
	 Y_4 &= & a_4^{-1} a_3^{-1} a_2^{-1} a_1^{-1} a_4 ^{-1}\alpha_n(a_3^{-1} 
a_2^{-1} a_1^{-1} a_4 ^{-1})\alpha_n^2(a_3^{-1} a_2^{-1} a_1^{-1} a_4 ^{-1}) \ldots\cr
	     &\, \vdots \\
	 Y_n &= & a_n^{-1} a_{n-1}^{-1} \ldots a_2^{-1} a_1^{-1} a_n^{-1}
\alpha_n(a_{n-1}^{-1} \ldots a_2^{-1} a_1^{-1} a_n^{-1})  \\
&&\alpha_n^2(a_{n-1}^{-1} 
\ldots a_2^{-1} a_1^{-1}        a_n^{-1}) \ldots \\
\end{array}
$$
As all $ \alpha_n(a_i) $ are positive and of length greater or equal
to $2$, it is easy to see that for any finite initial subword $X'$ of
$X_i$ (or of $Y_i$) the word $ \alpha_n(X') $ is again an initial
subword of $ X_i $ (or of $ Y_i $), and it is strictly longer. Hence
all the above words define attractive fixed points of $ \bdry
\alpha_n$, see \cite[Section I]{GJLL}. From the sign of the
exponents and from the initial letter it is easy to observe that they
are pairwise distinct.

We will show in \fullref{III.}  that $ \Fix(\alpha_n^{-1}) =
\Fix(\alpha_n) =\{1\}$.  Actually, we will show in Section V that
there are non-trivial fixed and not even periodic conjugacy classes of
$ \alpha_n\,$. Hence, in view of the inequality from \cite{GJLL}
stated in the Introduction, it could theoretically be that $\alpha_n$
or a power of $ \alpha_n $ has one more attractive fixed point on
$\bdry F_n\,$. However, for a proper power of $ \alpha_n $ this
couldn't be the case, as then the whole $ \alpha_n$--orbit of this
point would be fixed, thus giving more attractive fixed points than
the above inequality from \cite{GJLL} allows. For $ \alpha_n $ itself
this is excluded by the fact that there are only $ 2n - 1 $ total
occurences of any $ a_i $ in any reduced word $ \alpha_n(a_i)$, and
this number is an upper bound for the number of $
\Fix(\alpha_n)$--orbits of attracting fixed points in $ \bdry F_n\,$,
as has been shown in \cite[Theorem 2]{CL} (where one uses of
\cite[Proposition 1.1]{GJLL} for translation into our terminology).

\section{The repelling fixed points of $\bdry \alpha_n$}
\label{II.}

In order to compute the inverse of $\alpha_n$ we first define iteratively
$x_0 = a_1^{-1}$ and, for any $k$ with $0\leq k\leq n-1$, $x_{k+1} = 
a_{n-k} x_k^2\,.$ We now notice that:
$$
\begin{array}{lll}
\alpha_n(x_0) &= &{(a_1 a_2 \ldots a_n)}^{-1}, \\
		\alpha_n(a_n x_0) &= & a_n, \\
		\alpha_n(x_1) &= & {(a_1 a_2 \ldots a_{n-1})}^{-1}, \\
		\alpha_n(a_{n-1} x_1) &= & a_{n-1}, \\
				&\, \vdots \\
\end{array}
$$
$$
\begin{array}{lll}
		\alpha_n(x_{n-2}) &= & {(a_1 a_2)}^{-1}, \\
		\alpha_n(a_2 x_{n-2}) &= & a_2, \\
		\alpha_n(x_{n-1}) &= & a_1^{-1} \\
\end{array}
$$
Hence $\alpha_n^{-1}$ is given by $a_1 \mapsto x_{n-1}^{-1},\
		a_{n-k} \mapsto a_{n-k} x_k\ (k=0,\ldots,n-2).$
It is easy to see from the above computations that, if we replace the
basis element $a_1$ by its inverse $a_1^{-1} = x_0$, one obtains
$\alpha^{-1}$ again as positive automorphism, with respect to the new
basis $\{x_0, a_2, a_3, \ldots, a_n \}$.

In order to describe the attractive fixed points of $\alpha_n^{-1}$, we need 
some further notation. Define
\begin{align*}
y_k &= x_{n-1} x_k^{-1} x_0^{-1}\ \quad (k= 0, \ldots, n-2) \\
y &=  x_{n-1} x_0^{-1}\tag{$*$}\label{star}  \\
z &=  x_0^{-1} a_n^{-1} a_{n-1}^{-1} \ldots a_2^{-1}x_{n-1} 
\end{align*}
and notice that these are all positive words in the above defined new basis. 
We now define the infinite words
\begin{gather*}
X_k = a_{n-k} x_k \alpha_n^{-1}(x_k) \alpha_n^{-2}(x_k)\alpha_n^{-3}(x_k)\ldots\\
\tag*{\hbox{and}}
Y_k = a_{n-k} x_0^{-1} y_k^{-1} \alpha_n^{-1}(y_k^{-1})
\alpha_n^{-2}(y_k^{-1}) \alpha_n^{-3}(y_k^{-1}) \ldots
\end{gather*}
for $ k = 0, \ldots, n - 2\,$, as well as
\begin{gather*}
Y = x_0^{-1} y^{-1} \alpha_n^{-1}(y^{-1})
\alpha_n^{-2}(y^{-1}) \alpha_n^{-3}(y^{-1}) \ldots\\ 
\tag*{\hbox{and}}
Z=x_0^{-1} x_{n-1} \alpha_n^{-1}(z) \alpha_n^{-2}(z)\alpha_n^{-3}(z) \ldots
\end{gather*}
We first compute that these words are all fixed by $ \alpha_n^{-1} $:
For the $X_k$, the $Y_k$ and $Y$ this follows directly
from the given definition of $ \alpha_n^{-1}$, using in particular 
$\alpha_n^{-1}(x_0) = x_{n-1} $ and the definitions \eqref{star}. For $Z$ it
follows from the following computation:         
$$
\begin{array}{lll}
	\alpha_n^{-1}( x_0^{-1} x_{n-1})     &= & x_{n-1}^{-1}
\alpha_n^{-1}(a_2 \ldots a_n) x_{n-1} \alpha_n^{-1}(z) \\
		&= &x_{n-2}^{-2} a_2^{-1} \alpha_n^{-1}(a_2)
\alpha_n^{-1}(a_3 \ldots a_n) x_{n-1} \alpha_n^{-1}(z) \\
		&= &x_{n-2}^{-1} \alpha_n^{-1}(a_3 \ldots a_n) x_{n-1} \alpha_n^{-1}(z) \cr
		&= &x_{n-3}^{-2} a_3^{-1} \alpha_n^{-1}(a_3)
\alpha_n^{-1}(a_4 \ldots a_n) x_{n-1} \alpha_n^{-1}(z) \\
		&= &x_{n-3}^{-1} \alpha_n^{-1}(a_4 \ldots a_n) x_{n-1} \alpha_n^{-1}(z) \cr
		& \, \vdots \\
		&= &x_1^{-1} \alpha_n^{-1}(a_n) x_{n-1} \alpha_n^{-1}(z) \\
		&= & x_0^{-1} x_{n-1} \alpha_n^{-1}(z) \\
\end{array}
$$
The fact that all these words are $\alpha_n^{-1}$--attracting is a direct
consequence from the above observation that the words defined in \eqref{star}
are positive in the new basis.

Observe next that these infinite words are pairwise distinct:  The
words $ X_k $ and $ Z $ are all eventually positive and start with a
different letter (notice that the initial letter $ x_0^{-1} $ of $ Z $ is not
cancelled), and the same is true for the remaining ones, which are all
eventually negative.
Notice however that, for $ n = 2$, the two words $ Y_0 $ and $Y$ are 
related by the equation 
$$Y_0=a_2 x_0^{-1} a_2^{-1} x_0 Y,$$
and $ a_2 x_0^{-1} a_2^{-1} x_0 \in \Fix(\alpha_2) =\Fix(\alpha_2^{-1})$. 
In order to show that no such phenomenon occurs
for $ n \geq 3 $ it will be proved in Section III  that $ \Fix(\alpha_n) =
\Fix(\alpha_n^{-1}) $ is trivial. This implies, for $ n \geq 3$, that $
\bdry \alpha_n^{-1} $ has $ 2n $ attracting fixed points on $ \bdry F_n
$ which are all in distinct $ \Fix(\alpha_n^{-1})$--orbits.

\section{The fixed subgroup of $ \alpha_n$}
\label{III.}

In order to determine the fixed subgroup of $\alpha_n$ we use the
train track methods of  Bestvina and Handel \cite{BH}. As $\alpha_n$ is
positive, it follows that the standard rose $R_n$ with $n$ leaves
admits a train track representative $f\co R_n \to R_n$ of $\alpha_n\,$,
given simply by realizing the words $ \alpha_n(a_k) $ as reduced paths
in $ R_n\,$, with the unique vertex $\ast$ of $ R_n $ as initial and
terminal point.

Recall \cite{BH} that any conjugacy class $ [w] $ of $ F_n $ fixed
by the outer automorphism $\hat\alpha_n $ defined by $\alpha_n$ is
represented in the train track representative $ R_n $ by a loop $\gamma$ 
which is a concatenation of indivisible Nielsen paths (INP's).
Hence, in order to show that $ \Fix(\alpha_n) = \{1\} $, it suffices to show
that $ f $ does not have any INP's. For this purpose we first check for
illegal turns in $R_n$: A straight forward inspection, comparing
initial and terminal subwords of the $ \alpha_n(a_k) $ reveals that there
is only one illegal turn, given by $ (\bar a_1,\bar a_n)$.
Any  INP  in $ R_n $ must be of the form $\gamma_1 \gamma_2^{-1}$, such 
that $ \gamma_1 $ and $ \gamma_2$ are legal paths which both have terminal 
point at $\ast$ and define
there the above illegal turn. Hence one of the $\gamma_i\,$, say $\gamma_1$, 
ends in $a_1$, while the other one, $ \gamma_2\,$, ends in $a_n\,$. 
Their $f$--images have to be legal paths of the form $f(\gamma_1) = 
\gamma_1 \gamma_3 $ and $ f(\gamma_2) = \gamma_2\gamma_3$. Thus $ \gamma_3 $ 
ends in $ a_1 a_2\ldots a_n\,$. 

\medskip
{\bf Case 1}\qua $ \gamma_3 = a_1 a_2 \ldots a_n\,$.
It follows that the second to last letter in $ \gamma_1$, which preceeds 
$a_1$, has to have $\alpha_n$--image with terminal letter equal to $a_1$. But 
no such $ a_k $ exists!  

It follows that $\gamma_3$ ends in $a_n a_1 a_2 \ldots a_n\,$.
Then the second to the last letter in $ \gamma_1 $, 
preceding $ a_1 $, must be $a_n$ or $ a_1$.

\medskip
{\bf Case 2}\qua $\gamma_3 = a_n a_1 a_2 \ldots a_n\,$. 
Then in either of the last two subcases
 the last letter of $ \gamma_1 $ would have to be 
$a_{n-1}$, contradicting the above conditions.

It follows that the second to last letter of $ \gamma_2 $ must 
be $a_{n-1}\,$, and that $\gamma_3$ ends 
in $a_1 a_2 \ldots a_n a_1 a_2 \ldots a_n \,$.

\medskip
{\bf Case 3}\qua $ \gamma_3 = a_1 a_2 \ldots a_n a_1 a_2 \ldots a_n \,$. 
In this case the last letter of $ \gamma_2 $ is $ a_{n-1}$, again
contradicting the above conditions.

It follows that the second to last letter of $ \gamma_1 $ is not
$\alpha_n$ but $ a_1$, and the letter before must be $ a_{n-1}$.
At this point we know that $\gamma_3$ ends in 
$a_{n-1} a_1 a_2 \ldots a_n a_1 a_2 \ldots a_n \,$.

\medskip
{\bf Case 4}\qua $ \gamma_3 = a_{n-1} a_1 a_2 \ldots a_n a_1 a_2 \ldots a_n \,$.
Then the last letter of $ \gamma_1 $ would be $ a_{n-2}\,$,
contradicting the above conditions.
It follows that the third to the last letter in $ \gamma_2 $ is $ a_{n-2}\,$.
But then the only one possibility left is:

\medskip
{\bf Case 5}\qua $ \gamma_3 = a_1 a_2 \ldots a_{n-1} a_1 a_2 \ldots a_n a_1
a_2 \ldots a_n \,$.
Here the last letter of $ \gamma_1 $ would be $ a_{n-1}$,
contradicting the above conditions. 

\medskip
\noindent Notice that the argument in case $ 4 $ requires $ n \geq 3$.

This sweeps out all possibilities, and hence proves that there is no
INP  in $ R_n $ with respect to the train track map $f$, for $n\geq 3$. 

In Section V we will also consider the question of whether there
exists a path $\gamma_1 \gamma_2^{-1}$ in $R_n$ such that both
$\gamma_i$ are legal, and $f(\gamma_1) = \gamma_2 \gamma_3$, 
$f(\gamma_2) = \gamma_1 \gamma_3$.  The reader can check
without much difficulty, following precisely the same cases as
above, that such paths do not exist either.

\section{The irreducibility of $ \alpha_n $}
\label{IV.}

If $\alpha_n$ or a positive power of it were reducible, then there
would be a 
non-trivial 
proper free factor $F_m$ of $F_n$ which is left invariant
(up to conjugation)
by 
$\alpha_n^t$, for some $t\geq 1$. Passing over to an even higher power
and restricting possibly to another proper free factor of $ F_m $ we
can then assume that either $\alpha_n^t $ induces the trivial outer
automorphism on $F_m$, or else $\alpha_n^t\vert_{F_m}$ is
irreducible with irreducible powers. The first case is excluded by our
results in Section V, as then $ \alpha_n $ would have at least one 
non-trivial periodic conjugacy class. 
To rule out the second case we have to apply the following {\it
irreducibility test}, compare Bestvina and Handel \cite{BH} or
Lustig \cite{Lu2,Lu3}:

Let $f\co \Gamma \to \Gamma $ be a train track map in the sense
of \cite{BH}. Replace every vertex $v$ in $\Gamma $ by the 1--skeleton of
a $(k-1)$--simplex $\sigma(v)$, where $ k $ is the number of
edge 
gates at  $v$. 
(Recall that two edge germs $dE$ and $dE'$
raying out of a vertex $v$ belong to the same 
gate if and only if for some $t \geq 1$ the paths $f^t(E)$ and 
$f^t(E')$ have a non-tivial common initial subpath.)
This replacement is done by
glueing each such edge germ $ d E_i$ 
to the vertex $ v(d E_i) $ of $\sigma(v)$ which represents the gate to 
which $dE_{i}$ belongs. Now extend 
$f$ by mapping every edge $e$ of $\sigma(v) $ which connects $v(d E_i)$ 
to $v(d E_j)$
to the edge of $ \sigma (f(v)) $ which connects $ v(f(d E_i))
$ to $ v(f(d E_j))$. If $ f(d E_i) = f(d E_j)$, then map the
whole edge $e$ to the vertex $v(f(d E_i))$. Change the definition of $f$ 
along the edges of $\Gamma $ so that for any edge $ E_i $ of
$\Gamma $ the image is a reduced path in the new graph which agrees
with the old $ f(E_i) $ up to inserting precisely one of the ``new'' edges
(ie, the ones from the 1--skeletons of the $(k-1)$--simplices 
$\sigma(v)$) between
any two ``old'' edges which are adjacent in $ f(E_i)$. This defines a new
graph $\Gamma_1 $ and a new map $f_1\co \Gamma_1 \to\Gamma_1$.

We now omit from $\Gamma_1 $ all edges from the $(k-1)$--simplices $\sigma(v)$
which are 
not contained in any image $ f^t_1(E_i) $, for
any of the old edges $ E_i\,$ and $t \geq 1$.
Notice that this is done by a finite check, as $f_1$ is eventually
periodic on the new edges. The resulting graph $\Gamma_2 $ admits
a self map $f_2=f_1\vert_{\Gamma_2}\co \Gamma_2 \to 
\Gamma_2\,$,
and it is easy to see that $ f_2 $ inherits from $f$ the properties of a
train track map. Obviously there is a canonical map $\theta\co  \Gamma_2
\to \Gamma$, defined by 
the inclusion $\Gamma_{2} \subset \Gamma_{1}$ and subsequent
contraction of every new edge of $\Gamma_{1}$. Our
definitions give directly $f_2 \theta = \theta f$,
up to possibly reparametrizing $f$ along the edges.

\begin{proposition}[Irreducibility Criterion]
\label{IV.1}
Let $f\co \Gamma \to \Gamma $ 
be a train track map 
in the sense of {\rm\cite{BH}}, assume that its transition matrix is irreducible 
with irreducible powers,
and assume also that no $f_\ast^t$ with $t \geq 1$
fixes elementwise a
proper free factor of $\pi_1\Gamma\,$, up to conjugacy. 
Then 
$f_\ast \in Out(\pi_1\Gamma)$ is an irreducible automorphism 
with irreducible powers 
if and only if the induced map $ \theta_\ast\co 
\pi_1\Gamma_2 \to \pi_1\Gamma $ on the fundamental groups is surjective.
\end{proposition}

\begin{proof}
We freely use in this proof some of the $\R$--tree technology 
from \cite{GJLL} and from \cite[Sections 3--5]{Lu3},
from which we also borrow the 
terminology.  In particular, we consider the 
$\alpha$--invariant $\R$--tree $T$ with stretching factor $\lambda > 1$ 
which is given by the (up to scalar multiples) well defined 
Perron--Frobenius row eigen vector $\vec v_{*}$ of the transition matrix $M(f)$ of 
the train track map $f$. It comes with an $F_{n}$--equivariant 
map $i\co  \tilde \Gamma \to T$ which is isometric on edges (and more 
generally on legal paths), if the universal covering $\tilde \Gamma$ is 
provided with edge lengths as given by $\vec v_{*}$. Furthermore, there is a 
homothety $H\co  T \to T$ with stretching factor $\lambda$, which 
$\alpha$--twistedly commutes with the $F_{n}$--action:
It satisfies $\alpha(w) H = H w\co  T \to 
T$ for all $w \in F_{n}$.  If $\tilde f$ is the lift of $f$ to 
$\tilde \Gamma$ that also 
$\alpha$--twistedly commutes with the $F_{n}$--action, then $H$ 
and $\tilde f$ commute via $i$, ie, $H i = i \tilde f$.

We now assume that the map $\theta$ is not surjective, ie, some of 
the 1--skeleta $\sigma^{1}(v)$ of the simplices $\sigma(v)$ decompose into more than 
one connected component, when passing from $\Gamma_{1}$ to 
$\Gamma_{2}$. We pass over to a new graph $\Gamma_{3}$ in the 
following way:

For each of the simplices $\sigma(v)$ we reconnect the
connected components of $\sigma^{1}(v) \cap \Gamma_{2}$ 
by adding a new {\it center vertex} $c(v)$ to 
$\Gamma_{2}$ and 
connecting each connected component by a {\it central edge} to $c(v)$. 
We extend the train track map $f_{2}$ canonically to
obtain again a train track map $f_{3}\co  \Gamma_{3} \to \Gamma_{3}$, and 
a ``projection map'' $\theta_{3} \co  \Gamma_{3} \to \Gamma$ with 
$\theta_{3} f_3 = f \theta_{3}$ (up to isotopy within the images of 
single edges). By construction, $\theta_{3*}$ is now surjective. 
Note that the map $f_{3}$ respects the partition of the edges of 
$\Gamma_{3}$ into edges from $\Gamma_{2}$ and central edges.

We consider the universal covering $\tilde \Gamma_{3}$ and the 
canonical $F_{n}$--equivariant map $i_{3}\co  \tilde \Gamma_{3} \to T$ 
obtained from composing the lift of $\theta_{3}$ to $\tilde \Gamma$
with the above map 
$i$. Just as for $\Gamma$ we can also consider the transition matrix 
for $f_{3}$ and obtain in the analogous way  Perron--Frobenius edge 
lengths on $\tilde \Gamma_{3}$ to make the map $i_{3}$ edge isometric.  Of 
course, the resulting ``metric'' on $\tilde \Gamma_{3}$ is only a 
pseudo-metric, as all of the newly introduced central edges will get 
Perron--Frobenius length 0.

The usefulness of these ``invisible'' central edges however becomes 
immediately appearent:  Each {\it multipod} $Y(\tilde v)$, 
consisting of the lift to $\tilde \Gamma_{3}$ of a central 
vertex $c(v)$
with all adjacent central edges, is mapped by $i_{3}$ to a single point 
$Q(\tilde v) = i(\tilde v)$ in $T$ (here $\tilde v 
\in \tilde \Gamma$ is the corresponding lift of the vertex $v \in \Gamma$),
and the directions at this point are in canonical bijection 
(given by the map $i_{3}$) with the gates at $\tilde v$ and hence with 
the endpoints of the 
 multipod $Y(\tilde v)$. We can $F_{n}$--equivariantly replace the 
point $Q(\tilde v)$ by the multipod $Y(\tilde v)$,
where every direction of $T$ at 
$Q(\tilde v)$ is attached at the corresponding endpoint of $Y(\tilde v)$.
Again, we 
define the edge lengths throughout $Y(\tilde v)$ to be 0, so that metrically 
the resulting tree $T_{3}$ is the same as $T$. 

We now observe that the homothety $H_{3}\co  T_{3} \to T_{3}$, which $T_{3}$ 
canonically inherits from $H\co  T \to T$, can be shown to map on one 
hand the union $Y$ of all $Y(\tilde v)$ to itself, but similarly 
also its complement 
$T_{3} \smallsetminus Y$. This follows  from the commutativity 
equality $i_{3} \tilde  f_{3} = H_{3} i_{3}$ which is by the above 
construction
inherited from 
the equation $i \tilde f = H i$, and from the above observation that 
the subgraph $\Gamma_{2}$ of $\Gamma_{3}$, as well as its complement 
$\Gamma_{3} \smallsetminus \Gamma_{2}$, is kept invariant 
under the map $f_{3}$.
As a consequence, we can invert the situation, by considering the 
length function (also a row eigen vector of $M(f_{3})$ !) which 
associates length 1 to every edge of $Y$, and length 0 to all other 
edges, ie, contracting every complementary component of $Y$ in 
$T_{3}$ to a point. The resulting space $T_{3}^{*}$ is a simplicial 
$\R$--tree with trivial edge stabilizers,
and the map $H_{3}$ induces an isometry $H_{3}^{*}\co  
T_{3}^{*} \to T_{3}^{*}$ which 
$\alpha$--twistedly commutes with $\tilde f$ and 
commutes with the induced map $i_{3}^{*}\co  \tilde \Gamma_{3} \to 
T_{3}^{*}$. It follows that 
the Bass--Serre decomposition of $F_{n}$ associated to this simplicial 
tree is $\alpha$--invariant. In particular, the vertex groups of this 
decomposition give a non-empty collection of 
non-trivial proper free factors of $F_{n}$ which is 
$\alpha$--invariant, proving directly that $\alpha$ is not iwip.

To prove the converse implication of the theorem we can now invert
every step in the construction given above:  If $\alpha$ is 
reducible and no positive power fixes elementwise a free factor,
there exists a simplicial tree as $T_{3}^{*}$, and this 
tree is given (compare \cite{GJLL})
by a row eigenvector for the top stratum of some relative train track 
representative $f_{0}\co  \Gamma_{0} \to \Gamma_{0}$ of $\alpha$
as in \cite{BH}. Modifying this 
train track representative as in \cite{Lu2} to get a partial train track 
representative with Nielsen faces $\phi\co  \G \to \G$,
allows us, as above for the graph 
$\Gamma_{3}$, to represent simultaneously
both, the action on $T_{3}^{*}$ as well as 
that on $T$, by row eigen vectors of $M(\phi)$.
As a consequence one sees that the 
two trees come from a common ``refinement'', as given above
by the tree $T_{3}$:
Both, $T_{3}^{*}$ and $T$, are defined by a pseudo-metric on $T_{3}$
which is troughout zero, on vice-versa complementary $H_{3}$--invariant
subforests of 
$T_{3}$. We now consider again the originally given 
train track map $f\co  \Gamma \to \Gamma$ and its local ``blow-up'' $f_{1}\co  
\Gamma_{1} \to \Gamma_{1}$. The $H_{3}$--invariance of the two 
subforests translates (via the map $i_{1}\co  \tilde \Gamma_{1} \to T$
induced by $i$) into a non-trivial $f_{1}$--invariant subgraph of the 
union of the $\sigma^{1}(v)$, with invariant complement $\Gamma_{2}$. 
The connected components of this graph $\Gamma_{2}$ are in 1--1 
correcpondence with the $F_{n}$--orbits of the zero-valued subforests of 
$T_{3}$ defined by the row-eigen vector that gives $T_{3}^{*}$. Thus 
the non-triviality of the latter translates directly into the fact 
that the injection $\pi_{1} \Gamma_{2} \to \pi_{1} \Gamma_{1}$ is 
non-surjective. This finishes the proof.
\end{proof}

\begin{remark}
\label{IV.2}
\rm
The Irreducibility Criterion (\fullref{IV.1}) can alternatively be derived
as consequence of the theory of limit
laminations and their fundamental group, as developed in \cite{Lu1}. 
We sketch now an outline of the ``if''-direction:

Reducibility of $f_\ast$ or some positive power would give, as above 
explained for $ \alpha_n\,$, a proper free factor $ F_m $ of $F_n$ on which 
$f_\ast^t$ for some $ t \geq 1 $ acts as irreducible automorphism with 
irreducible powers. Such an automorphism has an expanding limit lamination 
$L$ with $\pi_1 L \subset F_m$. As $F_m$ embeds as free factor into 
$ F_n$, say $\rho\co F_m \to F_n$, we obtain $\pi_1(\rho (L)) \subset 
\rho (F_m)\ne F_n$ 
(compare \cite[Lemma 9.7]{Lu1}). 
On the other hand, it follows from 
the irreducibility of the transition matrix of $  f\co  \Gamma \to \Gamma $ that
there is only one
expanding limit lamination $ L^\infty(f)$.  Hence $ L^\infty(f) = \rho (L)$,
and we can apply  \cite[Korollar 7.7]{Lu1} with $ \tau = \Gamma_2$ 
(provided with an appropriate combinatorial labeling which reflects 
$\theta_\ast$) to deduce $ \pi_1(L^\infty(f)) = F_n$ from the surjectivity
of $\theta_\ast$, thus yielding a
contradiction to the above derived statement $\pi_1(\rho (L)) \subset 
\rho (F_m)\ne F_n$.
\end{remark}

In order to apply the Irreducibility Criterion \ref{IV.1} 
to the automorphism $\alpha_n$ as given in the Introduction,
we first compute directly from the definition of the $\alpha_n(a_i)$
that the transition matrix of $f$ is irreducible with irreducible powers.
Then
we have to replace the vertex $\ast$ by part of the 1--skeleton of a 
$(2n-1)$--simplex $\sigma = \sigma(\ast)$. 
We start with the 0--skeleton, and introduce only those edges 
of $\sigma$ which are
contained in the $f_1$--image of any of the old edges.  This gives two
connected components, where one of them contains only the vertex 
associated to the initial germ of $a_2$ and the one associated
to the terminal germ of
$a_1\,$, as well as a single new edge, say $\eta$, which connects them.
The other component contains all other vertices and a tree which connects
them (with the vertex asociated to the initial germ of  $a_1$ as ``root''
of the tree).
Now we have to fill in the forward $f_1$--orbit of the new edges
introduced so far. 
But the $f_1$--image
of $\eta$ connects the vertex of the initial germ of $a_2$ to that
of the terminal germ of $a_n\,$, so that in $\Gamma_2$ the subgraph
which belongs to the $(2n-1)$--simplex $\sigma$ is connected.  Hence $\theta_\ast$
is surjective.

\section{End of the proof and some remarks}
\label{V.}

In this section we consider the outer automorphism 
$\hat\alpha_n$ induced by 
$\alpha_n\,$, and its inverse $\hat \alpha_n^{-1}$. In \cite{GJLL} an index for 
outer automorphisms of $ F_n $ has been defined as follows: 
Two automorphisms of $F_n$ are
called {\it isogredient} (or 
in \cite{GJLL} {\it similar}), if they are conjugated in $\Aut(F_n)$
by an inner automorphism of $F_n\,$.  Let 
$\SS(\hat\alpha )$ denote the set of  isogredience classes $[\alpha']$ of
automorphisms $\alpha'$ inducing the outer automorphism
$\hat\alpha'=\hat\alpha$. 
We define 
$$
\ind (\hat\alpha ):=\sum_{[\alpha']\in \SS(\hat\alpha)}
\max(\rk(\Fix(\alpha'))+
{1\over 2}\# (\Fix^+(\bdry \alpha')/ \Fix(\alpha'))-1, 0).
$$ 
The main result of \cite[Theorem $1'$]{GJLL}, is equivalent to the inequality
$$
\ind(\hat\alpha)\ \le\ \ n -1 
$$ 
for all $\hat \alpha \in \Out (F)$.

Now, the outer automorphism $ \hat \alpha_n^{-1} $ has maximal
possible index $ n - 1$, all concentrated in one isogredience class of $\hat
\alpha_n^{-1}$, namely the one given by $ \alpha_n^{-1}$, and here
again all concentrated in the term 
${1\over 2}\#(\Fix^+(\bdry \alpha_n^{-1}) /
\Fix(\alpha_n^{-1}))$, 
which counts the number of the attractive fixed points
at $ \bdry F_n\,$, as the fixed subgroup of $\alpha_n^{-1}$ is trivial.

We remark at this point that, if $X$ and $Y$ are infinite words, both
fixed by an automorphism $\alpha$, and $w X = Y$ for some $w \in F_n\,$,
then it follows
from an elementary combinatorial case checking that
$\alpha(w) = w$.  Hence we know
that the index contribution
of the  2n attracting fixed points of $\alpha_n^{-1}$ computed in Section II will 
be the same for all positive powers of $\alpha_n^{-1}$: 
On the other hand
(compare \cite{BH}), a fixed
non-trivial conjugacy class for some $\alpha_n^{-t}, \, t \ge 1$, will be represented 
by a concatenation of INP's of a train track representative of $\alpha_n^{-t}$,
which would contribute at least one infinite attracting fixed word in the same
isogredience class of $\hat\alpha_n^{-t}$ 
which fixes the non-trivial word read off from the concatenation of INP's. 
Hence we would get another
positive index contribution for $\hat\alpha_n^{-t}$, in contradiction to the above
inequality for the index.  Hence  $\alpha^{-1}$ and thus $\alpha$ can not have 
non-trivial periodic conjugacy classes.

\begin{remark}\label{V.1}
\rm
Outer automorphisms of $F_n$ with a positive power of index $n-1$ 
which are not geometric (ie, they are not induced by a 
homeomorphisms of a
surface with boundary) have been termed 
{\it para-geometric} in \cite[Section VI]{GJLL}, as, just 
as for geometric automorphisms, their action on 
any 
forward
limit tree is
geometric (in the sense of Gaboriau and Levitt
\cite{GL}).
Guirardel \cite{Guir} shows that if for an iwip automorphism $\alpha$
both, the (uniquely determined)
forward and the backward limit trees are geometric 
($\Longleftrightarrow \ind(\hat 
\alpha^t) = \ind(\hat 
\alpha^{-t}) = n-1$ for some sufficiently large $t \geq 1$), then 
$\alpha$ is geometric.
\end{remark}

To see whether the irreducible (and non-geometric !)
automorphism $\alpha_n$ itself is
parageometric or 
not we can either apply the result of Guirardel \cite{Guir} quoted 
in the Introduction, or else apply direct arguments which seem 
interesting in their own right, as they are typical for similar 
computations for many other automorphisms:

We will compute the index of $\alpha_{n}$ and
that of its positive powers: From the previous
sections we know already that there is one isogredience class, given by 
$\alpha_n\,$, which contributes $0$ from $\Fix(\alpha_n) $ and $2n-1$ from 
$\Fix^+(\bdry \alpha_n)$, adding properly up to an index
contribution of $\rk(\Fix(\alpha_n)) + {1\over 2}\#(\Fix^+(\bdry \alpha_n)/
\Fix(\alpha_n))-1 = n - {3\over 2}$. Hence the only possibility for 
$\hat\alpha_n $ to have index $n-1$ is if there is another isogredience class,
represented by some automorphism $ \alpha '_n\,$, with index
contribution of ${1\over 2}$. As we have shown above that there is no
non-trivial
conjugacy class fixed by $ \hat \alpha_n \,$, the only possibility is that this
$\alpha'_n $ has $3$ attracting fixed points at $\bdry F_n\,$. In this
case the train track representative $f\co R_n \to R_n $ of 
$\hat \alpha_n$ has to have either another fixed point with $3$ distinct 
fixed directions (= edge germs), but this is not the case as $R_n$ has only
one vertex. Otherwise there must be two distinct fixed points in $R_n\,$, each 
with $2$ fixed directions, and they are connected by an 
INP. But we have shown in Section III that INP's do not exist for 
$f\co R_n \to R_n\,$. Hence it follows that 
$\ind(\hat \alpha_n) = n - {3 \over 2}$. 

The same arguments apply to all positive powers of $ \alpha_n\,$, except that
we have to rule out also the possibility of {\it periodic INP's}: If there is an
INP for $\alpha_n^t$ which is not an INP for $\alpha_n\,$, then its whole 
$\alpha_n$--orbit consists of INP's for $\alpha_n^t \,$. As this would immediately
give a too large index for $\hat\alpha_n^t$ if the orbit consists of
more than one INP, the only possibility left is that there is
an INP for $\alpha_n^2\,$, and $\alpha$ fixes this path too, but reverses
its orientation. But this possibility has been ruled out in the last paragraph
of Section III. Thus $\alpha_{n}$ is not parageometric (and also not 
geometric).

To finish this discussion, we would like to point out a subtle point
in which the non-geometric and non-parageometric
$\alpha_n$ and the parageometric $ \alpha_n^{-1}$ 
differ, which is characteristic for their classes:

For any parageometric automorphism (as $\hat\alpha_n^{-1}$) there is a {\it stable} 
train track representative
with a single illegal turn, namely the one at the tip of the unique INP, see 
\cite{BH}.
If we keep folding at this illegal turn, we get iteratively smaller and smaller 
copies of the train track, thus realizing the inverse of the train track map by
``continuous iterated folding'' (compare \cite{LL}).  Now, if we 
consider the train track representative
$f \co R_n \to R_n$ of the (non-parageometric !) $\hat \alpha_n$ 
from Section III, there is also a single illegal turn, and
if we keep folding there, it turns out that 
this will always be the case, as  
there will never appear any other illegal turn.
Thus the
situation looks remarkably similar to that in the parageometric case. There is, 
however, an interesting difference: 
If we trace in $R_n$ (or rather in the universal covering $\tilde 
R_{n}$) the two ``paths'' which will be folded
together in this iterative folding procedure, we will see that these are not 
two continuous arcs with the same initial point (as would be true 
in the parageometric case,
given there by the two subarcs of the INP which meet at
the illegal turn), but much rather there will be lots of (indeed 
infinitely many !) discontinuities in 
these ``paths''.  Each of these will disappear eventually in the folding process,
but initially they are present.  We believe that in these discontinuities
the core information is encoded, for a
geometric understanding of the gap between the maximal index of 
a positive power of the automorphism and the above upper bound $n-1$.

\bibliographystyle{gtart}
\bibliography{link}

\begin{thebibliography}{}
\providecommand\bibmarginpar{\leavevmode\marginpar}
\def\urlstyle#1{{\tt #1}}

\bibitem{BH}
\textbf{M Bestvina}, \textbf{M Handel},
  \href{http://dx.doi.org/10.2307/2946562} {\emph{Train tracks and
  automorphisms of free groups}}, Ann. of Math. $(2)$ 135 (1992) 1--51
  \xox{MR}{1147956}

\bibitem{CL}
\textbf{M\,M Cohen}, \textbf{M Lustig},
  \href{http://dx.doi.org/10.1007/BF01393699} {\emph{On the dynamics and the
  fixed subgroup of a free group automorphism}}, Invent. Math. 96 (1989)
  613--638 \xox{MR}{996557}

\bibitem{Co}
\textbf{D Cooper}, \href{http://dx.doi.org/10.1016/0021-8693(87)90229-8}
  {\emph{Automorphisms of free groups have finitely generated fixed point
  sets}}, J. Algebra 111 (1987) 453--456 \xox{MR}{916179}

\bibitem{GJLL}
\textbf{D Gaboriau}, \textbf{A Jaeger}, \textbf{G Levitt}, \textbf{M Lustig},
  \href{http://dx.doi.org/10.1215/S0012-7094-98-09314-0} {\emph{An index for
  counting fixed points of automorphisms of free groups}}, Duke Math. J. 93
  (1998) 425--452 \xox{MR}{1626723}

\bibitem{GL}
\textbf{D Gaboriau}, \textbf{G Levitt},
  \href{http://www.numdam.org/item?id=ASENS_1995_4_28_5_549_0} {\emph{The rank
  of actions on {$\mathbf{R}$}-trees}}, Ann. Sci. \'Ecole Norm. Sup. $(4)$ 28
  (1995) 549--570 \xox{MR}{1341661}

\bibitem{GLL}
\textbf{D Gaboriau}, \textbf{G Levitt}, \textbf{M Lustig}, \emph{A
  dendrological proof of the {S}cott conjecture for automorphisms of free
  groups}, Proc. Edinburgh Math. Soc. $(2)$ 41 (1998) 325--332
  \xox{MR}{1626429}

\bibitem{Ge}
\textbf{S\,M Gersten}, \emph{Fixed points of automorphisms of free groups},
  Adv. in Math. 64 (1987) 51--85 \xox{MR}{879856}

\bibitem{GT}
\textbf{R\,Z Goldstein}, \textbf{E\,C Turner},
  \href{http://dx.doi.org/10.1112/blms/18.5.468} {\emph{Fixed subgroups of
  homomorphisms of free groups}}, Bull. London Math. Soc. 18 (1986) 468--470
  \xox{MR}{847985}

\bibitem{Guir}
\textbf{V Guirardel}, \href{http://dx.doi.org/10.1016/j.ansens.2005.11.001}
  {\emph{C\oe ur et nombre d'intersection pour les actions de groupes sur les
  arbres}}, Ann. Sci. \'Ecole Norm. Sup. $(4)$ 38 (2005) 847--888
  \xox{MR}{2216833}

\bibitem{HM}
\textbf{M Handel}, \textbf{L Mosher},
  \href{http://dx.doi.org/10.1090/S0002-9947-07-04065-2} {\emph{Parageometric
  outer automorphisms of free groups}}, Trans. Amer. Math. Soc. 359 (2007)
  3153--3183 \xox{MR}{2299450}

\bibitem{LL}
\textbf{J Los}, \textbf{M Lustig}, \emph{The set of train track representatives
  of an irreducible free group automorphism is contractible}
\ Available at \setbox0\hbox{\makeatletter\@url
{http://www.crm.es/Publications/04/pr606.pdf}}
\href{http://www.crm.es/Publications/04/pr606.pdf}
{\unhbox0}

\bibitem{Lu1}
\textbf{M Lustig}, \emph{Automorphismen von freien Gruppen},
  Habilitationsschrift, Ruhr-Universit\"at Bochum (1992)

\bibitem{Lu2}
\textbf{M Lustig}, \emph{Structure and conjugacy for automorphisms of free
  groups I, II}, Max-Planck Institut f\"ur Mathematik, Preprint Series 2000 (No
  130) and 2001 (No 4)
\ Available at \setbox0\hbox{\makeatletter\@url
{http://www.mpim-bonn.mpg.de}}
\href{http://www.mpim-bonn.mpg.de}
{\unhbox0}

\bibitem{Lu3}
\textbf{M Lustig}, \emph{Conjugacy and centralizers for iwip automophisms of
  free groups}, from: ``Geometric Group Theory'', Trends in Mathematics,
  Birkh\"auser (2007)  197--224

\bibitem{Pa}
\textbf{F Paulin}, \href{http://dx.doi.org/10.1016/S0012-9593(97)89917-X}
  {\emph{Sur les automorphismes ext\'erieurs des groupes hyperboliques}}, Ann.
  Sci. \'Ecole Norm. Sup. $(4)$ 30 (1997) 147--167 \xox{MR}{1432052}

\end{thebibliography}

\end{document}